\documentclass[12pt]{article}
\usepackage{amsmath,amssymb,amsfonts}
\usepackage{epsfig}
\usepackage[]{hyperref}
 \hypersetup{colorlinks=true, citecolor=black,%
                 filecolor=black,%
                 linkcolor=black,%
                 urlcolor=black,  pdftex}

\newtheorem{theorem}{Theorem}[section]
\newtheorem{lemma}[theorem]{Lemma}

\newtheorem{proposition}[theorem]{Proposition}

\newtheorem{corollary}[theorem]{Corollary}

\newtheorem{claim}{Claim}

\newenvironment{proof}{\noindent\emph{Proof.}\hspace{.25em}}{\hspace*{\fill}
$\Box$\newline}

\def\pfsp{\hskip 1em}

\newcommand{\pfcl}[1]{\noindent{\it Proof\,:}\pfsp #1 \QED \smallskip}

\def \QED {\hfill $\triangle$}      
\def \QD1 {\hfill $\spadesuit$}

\newcommand{\DF}[1]{{\bf #1\/}}

\newcommand{\sm}{\setminus}

\newcommand{\De}{\Delta}

\newcommand{\col}{{\rm col}}

\newcommand{\cn}{\chi}

\newcommand{\cV}{{\cal V}}

\newcommand{\cF}{{\cal F}}

\newcommand{\nato}{\mathbb{N}_0}

\numberwithin{equation}{section}

\begin{document}
\title{\bf Vertex partition of hypergraphs and maximum degenerate subhypergraphs}

\author{{{
Thomas Schweser}\thanks{The authors thank the Danish Research Council for support through the program Algodisc.}
\thanks{
Technische Universit\"at Ilmenau, Inst. of Math., PF 100565, D-98684 Ilmenau, Germany. E-mail
address: thomas.schweser@tu-ilmenau.de}}
\and{{Michael Stiebitz}\footnotemark[1]
\thanks{
Technische Universit\"at Ilmenau, Inst. of Math., PF 100565, D-98684 Ilmenau, Germany. E-mail
address: michael.stiebitz@tu-ilmenau.de}}
}

\date{}
\maketitle

\begin{abstract}
\noindent In 2007 Matamala proved that if $G$ is a simple graph with maximum degree $\De\geq 3$ not containing $K_{\De+1}$ as a subgraph and $s, t$ are positive integers such that $s+t \geq \De$, then the vertex set of $G$ admits a partition $(S,T)$ such that $G[S]$ is a maximum order $(s-1)$-degenerate subgraph of $G$ and $G[T]$ is a $(t-1)$-degenerate subgraph of $G$. This result extended earlier results obtained by Borodin, by Bollob\'as and Manvel, by Catlin, by Gerencs\'{e}r and by Catlin and Lai. In this paper we prove a hypergraph version of this result and extend it to variable degeneracy and to partitions into more than two parts, thereby extending a result by Borodin, Kostochka, and Toft.
\end{abstract}



\section{Introduction and main results}

The paper deals with partition of hypergraphs into a fixed number of subhypergraphs so that each part satisfies a certain degree condition. Graphs and hypergraphs considered in this paper may have parallel edges, but no loops. We will mainly use the notation from the paper \cite{SchweserS2017}. Let $G$ be a hypergraph. As usual, we denote by $V(G)$ the \DF{vertex set} of $G$ and by $E(G)$ the \DF{edge set} of $G$. For a vertex $v$ of $G$, let $E_G(v)$ denote the set of edges of $G$ that are \DF{incident} with $v$ in $G$. Then $d_G(v)=|E_G(v)|$ is the \DF{degree} of $v$ in $G$, and $\De(G)=\max_{v\in V(G)}d_G(v)$ is the \DF{maximum degree} of $G$. Given two vertices $v \neq w$ of a hypergraph $G$, a $(v,w)$-\DF{hyperpath of length} $q$ in $G$ is a sequence
$(v_1, e_1, v_2, e_2, \ldots, v_q, e_q, v_{q+1})$ of distinct vertices
$v_1, v_2, \ldots, v_{q+1}$ of $G$ and distinct edges
$e_1, e_2, \ldots, e_q$ of $G$ such that $v_1=v$, $v_{q+1}=u$, and $e_i \in E_G(v_i) \cap E_G(v_{i+1})$
 for $i \in \{1, 2, \ldots, q\}$. By $\text{dist}_G(v,w)$ we denote the length of a shortest $(v,w)$-hyperpath in $G$. The hypergraph $G$ is connected if for any two vertices $v,w$ of $G$ there is a $(v,w)$-hyperpath in $G$. A (connected) \DF{component} of a nonempty hypergraph $G$ is a maximal connected subhypergraph.

For a vertex set $X\subseteq V(G)$, we denote by $G[X]$ the subhypergraph of $G$ induced by $X$, that is, the hypergraph whose vertex set is $X$ and whose edges are all edges of $G$ that are incident only to vertices in $X$. Furthermore, $G-X=G[V(G)\sm X]$. If $X=\{v\}$ is a singleton, then we also write $G-v$ instead of $G-X$. A subgraph $H$ of $G$ is an \DF{induced subhypergraph} of $G$ if $V(H)\subseteq V(G)$ and $H=G[V(H)]$. If $H$ is an induced subhypergraph of $G$ and $v\in V(G)$, then $H+v=G[V(H) \cup \{v\}]$. A \DF{partition} of a hypergraph $G$ is a sequence of induced subhypergraphs of $G$ (possibly empty) such that each vertex belongs to exactly one hypergraph of the sequence.

The first result dealing with partition of graphs under degree constraints was obtained in 1966 by Lov\'asz \cite{Lovasz66}. He proved that if $G$ is a simple graph and $d_1, d_2, \ldots, d_p$ are non-negative integers such that $d_1+d_2+\cdots + d_p\geq  \De(G)-p+1$, then there is a partition $(G_1,G_2, \ldots , G_p)$ of $G$ such that $\De(G_i)\leq d_i$ for $1 \leq i \leq p$. It is easy to see that Lov\'asz's partition result also holds for hypergraphs; one can apply the same simple argument as in Lov\'asz's original proof. Csisz\'ar and K\"orner \cite{ZsiszarK81} used Lov\'asz's argument to derive a continues version of his partition result for edge weighted graphs; they used this result for proving coding theorems. A variable version of Lov\'asz's result was obtained in 1977 by Borodin and Kostochka \cite{BorodinK77}. They proved that if $G$ is a simple graph and $f_1, f_2, \ldots, f_p:V(G)\to \nato$ are $p$ vertex functions such that $f_1(v)+f_2(v)+\cdots + f_p(v)\geq d_G(v)-p+1$ for all $v\in V(G)$, then $G$ has a partition $(G_1,G_2, \ldots, G_p)$ such that $d_{G_i}(v)\leq f_i(v)$ whenever $v\in V(G_i)$ and $i\in \{1,2, \ldots, p\}$. Also this result can easily be extended to hypergraphs.

The \DF{coloring number} $\col(G)$ of a non-empty hypergraph $G$ is $1$ plus the maximum minimum degree of the subhypergraphs of $G$. If $G$ is the empty hypergraph  (that is, $V(G)=E(G)=\varnothing$), we set $\text{col}(G)=0$. So if $d$ is a non-negative integer, then $\col(G)\leq d$ if and only if every non-empty subhypergraph of $G$ contains a vertex of degree at most $d-1$. In particular, $\col(G) \leq 0$ if and only if $G$ is empty and $\col(G)\leq 1$ if and only if $G$ is edgeless.

Borodin \cite{Borodin76} and, independently, Bollob\'as and Manvel \cite{BollobasM79} proved that if $G$ is a connected simple graph with maximum degree $\De\geq 3$ different from $K_{\De+1}$ and $d_1, d_2, \ldots, d_p$ are positive integers such that $d_1+d_2+\cdots + d_p\geq \De$, then $G$ has a partition $(G_1, G_2, \ldots, G_p)$ such that $\col(G_i)\leq d_i$ for $1\leq i \leq p$. The famous theorem of Brooks \cite{Brooks49}, saying that a connected simple graph with maximum degree $\De\geq 3$ satisfies $\cn(G)\leq \De+1$ and equality holds if and only if $G=K_{\De+1}$, follows from the former result by taking $p=\De$ and $d_i=1$ for $1\leq i \leq p$. Here $\cn(G)$ denotes the \DF{chromatic number} of $G$, that is, the least integer $p$ such that $G$ has a partition into $p$ edgeless subgraphs. The cases of point aboricity (which correspond to $d_1=d_2= \cdots =d_p=2$), and of point partiton numbers in general (which corresponds to $d_1=d_2= \cdots =d_p$) were solved by Kronk and Mitchem \cite{KronkM75}, and Mitchem \cite{Mitchem77}. The point partition number was introduced by Lick and White \cite{LickW72}.

A variable version of the result by Borodin, respectively Bollob\'as and Manvel, was obtained in 2000 by Borodin, Kostochka, and Toft \cite{BorodinKT2000} for simple graphs. Schweser and Stiebitz \cite{SchweserS2017} extended this result to hypergraphs. Let $G$ be a hypergraph, and let $h:V(G) \to \nato$ be a function from the vertex set of $G$ into the set of non-negative integers. The hypergraph $G$ is said to be \DF{strictly $h$-degenerate} if every non-empty subhypergraph $H$ of $G$ has a vertex $v$ such that $d_H(v)\leq h(v)-1$. Note that if $h(v)\equiv d$ is the constant function, then $G$ is strictly $h$-degenerate if and only if $\col(G)\leq d$. Degeneracy of graphs was introduced by Lick and White \cite{LickW70}. The hypergraph $G$ is called \DF{$h$-regular} if $d_G(v)=h(v)$ for all $v\in V(G)$. 
%

Let $G$ be an arbitrary hypergraph. A function $f:V(H) \to \nato^p$ is called a \DF{vector function} of $G$. By $f_i$ we name the $i$th coordinate of $f$, i.e., $f=(f_1,f_2, \ldots, f_p)$. The set of all vector functions of $G$ with $p$ coordinates is denoted by $\cV_p(G)$. For $f\in \cV_p(H)$, an \DF{$f$-partition} of $G$ is a partition $(G_1,G_2, \ldots, G_p)$ of $G$ such that $G_i$ is strictly $f_i$-degenerate for all $i\in \{1,2, \ldots, p\}$. If the hypergraph $G$ admits an $f$-partition, then $G$ is said to be \DF{$f$-partitionable}.

Recall that a \DF{block} of a hypergraph $G$ is a maximal connected subhypergraph of $G$ without a separating vertex. If $G$ itself has no separating vertex, $G$ is said to be a block. For a simple graph $H$ and an integer $t\geq 1$, let $G=tH$ denote the graph obtained from $H$ by replacing each edge of $H$ by $t$ parallel edges.

Let $G$ be a connected hypergraph and let $f \in \mathcal{V}_p(G)$ be a vector-function for some integer $p \geq 1$. We say that $(G,f)$ is a \textbf{hard pair} if one of the following four conditions holds.

\begin{itemize}
\item[(1)] $G$ is a block and there exists an index $j \in \{1,2,\ldots,p\}$ such that
$$f_i(v)=
\begin{cases}
d_H(v) & \text{if } i=j,\\
0 & \text{otherwise}
\end{cases}$$
for all $i \in \{1,2,\ldots,p\}$ and for each $v \in V(G)$. 
\item[(2)] $G=tK_n$ for some $t \geq 1, n \geq 3$ and there are integers $n_1,n_2,\ldots,n_p \geq 0$ with at least two $n_i$ different from zero such that $n_1 + n_2 + \ldots + n_p=n-1$ and that
$$f(v)=(tn_1,tn_2,\ldots,tn_p)$$
for all $v \in V(G)$.
\item[(3)] $G=tC_n$ with $t \geq 1$ and $n \geq 5$ odd and there are two indices $k \neq \ell$ from the set $\{1,2,\ldots,p\}$ such that
$$f_i(v)=
\begin{cases}
t & \text{if } i \in \{k,\ell\}, \\
0 & \text{otherwise}
\end{cases}
$$ for all $i \in \{1,2,\ldots,p\}$ and for each $v \in V(G)$. In this case, we say that $G$ is a block of type \textbf{(C)}.
\item[(4)] There are two disjoint hard pairs $(G^1,f^1)$ and $(G^2,f^2)$ with $f^1 \in \mathcal{V}_p(G^1)$ and $f^2 \in \mathcal{V}_p(G^2)$ such that $G$ is obtained from $G^1$ and $G^2$ by merging two vertices $v^1 \in V(G_1)$ and $v^ 2 \in V(G_2)$ to a new vertex $v^*$ (see Figure~\ref{figure_merging}). Furthermore, it holds
$$f(v)=
\begin{cases}
f^1(v) & \text{if } v \in V(H_1) \sm \{v^1\}, \\
f^2(v) & \text{if } v \in V(H_2) \sm \{v^2\}, \\
f^1(v^1) + f^2(v^2) & \text{if } v=v^*
\end{cases}$$
for all $v \in V(G)$. In this case we say that $(G,f)$ is obtained from $(G^1,f^1)$ and $(G^2,f^2)$ by merging $v^1$ and $v^2$ to $v^*$.
\end{itemize}

\begin{figure}[ht]
\centering
\includegraphics[width=.8\textwidth]{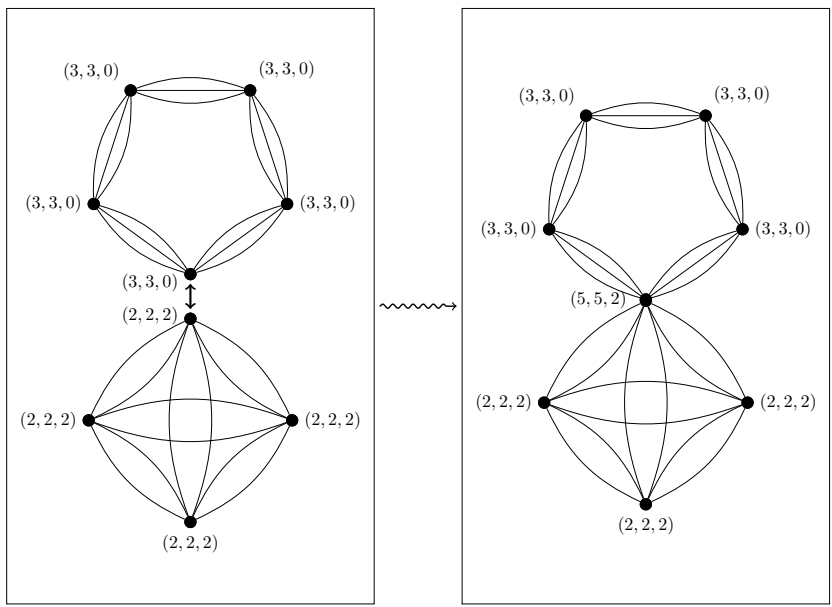}
\caption{Merging two hard pairs.}
\label{figure_merging}
\end{figure}
Note that a hypergraph $G$ is $f$-partitionable if and only if each component of $G$ is $f$-partitionable. Thus, it is sufficient to consider only connected hypergraphs. The next result was proved by Schweser and Stiebitz \cite{SchweserS2017}; for the class of simple graphs it was proved in 2000 by Borodin, Kostochka and Toft \cite{BorodinKT2000}.

\begin{theorem}
\label{Theorem:Hauptsatz:alt}
Let $G$ be a connected hypergraph and let $f\in \cV_p(G)$ be a vector function with $p\geq 1$ such that $f_1(v)+f_2(v)+\cdots +f_p(v)\geq d_G(v)$ for all $v\in V(G)$. Then, $G$ is $f$-partitionable if and only if $(G,f)$ is not a hard pair.
\end{theorem}

On the one hand, Theorem~\ref{Theorem:Hauptsatz:alt} is a strengthening of the result by Borodin, respectively Bollob\'as and Manvel. On the other hand, as explained in \cite{SchweserS2017}, Theorem~\ref{Theorem:Hauptsatz:alt} implies several well known result about colorings and list-colorings of graphs, respectively hypergraphs; in particular, the characterization of degree choosable graphs obtained by Erd\H{o}s, Rubin, and Taylor \cite{ERT79} and the characterization of degree choosable hypergraphs given by Kostochka, Stiebitz, and Wirth \cite{KostochkaSW96}. The special case when $p=\De(G)$ and $f_i(v)=1$ for all $v\in V(G)$ and $1\leq i \leq p$ yields a Brooks-type result for hypergraphs which was obtained by Jones \cite{Jones75}.

In 2007 Matamala \cite{Matamala2007} obtained another strengthening of the result by Borodin, respectively Bollob\'as and Manvel. He proved that if $G$ is a simple graph with maximum degree $\De\geq 3$ not containing a $K_{\De+1}$ as a subgraph and $d_1, d_2$ are positive integers with $d_1+d_2\geq \De$, then there is a partition $(G_1,G_2)$ of $G$ such that $G_1$ is a maximum order induced subgraph with $\col(G_1)\leq d_1$ and $\col(G_2)\leq d_2$. This result improves earlier results obtained by Catlin \cite{Catlin79}, Gerencs\'er \cite{Gerencer65}, and Catlin and Lai \cite{CatlinL95}. Catlin and Gerencs\'{e}r proved that if $G$ is a simple graph with maximum degree $\De\geq 3$ not containing a $K_{\De+1}$, then $G$ has a $\Delta$-coloring in which one color class is a maximum independent set. The main result of this paper is the following generalization of Matamala's theorem.

\begin{theorem}
\label{Theorem:Hauptsatz:neu}
Let $G$ be a hypergraph and let $f\in \cV_p(G)$ be a vector function of $G$ with $p\geq 2$ such that $f_1(v)+f_2(v)+\cdots +f_p(v)\geq d_G(v)$ for all $v\in V(G)$. Furthermore, assume that if $G'$ is a component of $G$, then $(G',f)$ is not a hard pair. Then, there is a partition $(G_1,G_2, \ldots, G_p)$ of $G$ such that $G_1$ is a maximum order strictly $f_1$-degenerate subhypergraph of $G$, and for $i \in \{2,3,\ldots,p-1\}$, the hypergraph $G_i$ is a maximum order strictly $f_i$-degenerate subhypergraph of $G-(V(G_1) \cup V(G_2) \cup \cdots \cup V(G_{i-1}))$.
\end{theorem}

\section{Proof of Theorem~\ref{Theorem:Hauptsatz:neu}}
\begin{proposition} \label{prop:strictly-h}
Let $G$ be a hypergraph, and let $f\in \cV_p(G)$ be a vector function of $G$ with $p\geq 1$, and let
$h :V(G) \to \nato$ be the function with $h(v)=f_1(v)+f_2(v)+\cdots + f_p(v)$ for all $v\in V(G)$. If $G$ is strictly $h$-degenerate, then $G$ is $f$-partitionable.
\end{proposition}
\begin{proof}
The proof is by induction on the order $n=|G|$ of $G$. If $n=1$, then $V(G)=\{v\}$ consists of only one vertex and, as $G$ is strictly $h$-degenerate, $0=d_G(v)< h(v)= f_1(v) + f_2(v) + \ldots + f_p(v)$, which implies that there is an index $i \in \{1,2,\ldots,p\}$ such that $f_i(v) > 0$. Setting $G_i=G[\{v\}]$ and $G_j= \varnothing$ for $j \neq i$ from $\{1,2,\ldots,p\}$ then gives us the $f$-partition $(G_1,G_2,\ldots,G_p)$ as claimed. Now assume $n \geq 2$. Since $G$ is strictly $h$-degenerate, there is a vertex $v \in V(G)$ with $d_G(v) < h(v)$. Clearly, $G-v$ is strictly $h$-degenerate, and so $G-v$ admits an $f$-partition $(G_1,G_2,\ldots,G_p)$ (by induction hypothesis). As $d_G(v) < h(v) = f_1(v) + f_2(v) + \ldots + f_p(v)$, it follows from the pigeonhole principle that $d_{G_i}(v) < f_i(v)$ for some $i \in \{1,2,\ldots,p\}$, say for $i=1$. Then, $G_1+v$ is strictly $f_1$-degenerate and so $(G_1+v,G_2,\ldots,G_p)$ is an $f$-partition of $G$, as claimed. This completes the proof.
\end{proof}

\begin{proposition} \label{prop:strict=partitionable}
Let $G$ be a connected hypergraph, and let $f\in \cV_p(G)$ be a vector function of $G$ with $p\geq 1$ such that $f_1(v) + f_2(v) + \ldots f_p(v) \geq d_G(v)$ for all $v \in V(G)$.  If $G$ is not $f$-partitionable, then $f_1(v) + f_2(v) + \ldots + f_p(v) = d_G(v)$ for all $v \in V(G)$.
\end{proposition}
\begin{proof} 
Let $h:V(G) \to \mathbb{N}_0$ with $h(v)=f_1(v) + f_2(v) + \ldots + f_p(v)$ for all $v \in V(G)$. Then, $d_G(v) \leq h(v)$ for all $v \in V(G)$. Assume that there is a vertex $u \in V(G)$ with $d_G(u) < f_1(v) + f_2(v) + \ldots + f_p(v)=h(u)$. As $G$ is connected, it then follows that $G$ is strictly $h$-degenerate. Proposition~\ref{prop:strictly-h} then implies that $G$ admits an $f$-partition, a contradiction.
\end{proof}

\begin{lemma}
Let $G$ be a hypergraph and let $f\in \cV_p(G)$ be a vector function of $G$ with $p\geq 2$ such that $f_1(v)+f_2(v)+\cdots +f_p(v)\geq d_G(v)$ for all $v\in V(G)$. If $G$ is $f$-partitionable, then there is an $f$-partition $(G_1,G_2, \ldots, G_p)$ of $G$ such that $G_1$ is a maximum order strictly $f_1$-degenerate subhypergraph of $G$.
\label{lemma:partitio=better}
\end{lemma}
\begin{proof}
The proof is by reductio ad absurdum. Let $\cF$ denote the set of tuples $(G_1,G_2, \ldots, G_p,G_1^*,G_2^*)$  such that $(G_1, G_2, \ldots, G_p)$ is an $f$-partition of $G$, $G_1^*$ is a maximum order strictly $f_1$-degenerate subhypergraph of $G$, and $G_2^*=G\sm V(G_1^*)$. Furthermore, let $f'=(f_2, f_3, \ldots, f_p)$ and let $h=f_2+f_3+ \cdots + f_p$. By assumption, $G$ has an $f$-partition. Clearly, $G$ has a maximum order strictly $f_1$-degenerate subhypergraph. Hence, $\cF$ is non-empty. 

\begin{claim} \label{claim_tuple}
Let $(G_1,G_2, \ldots, G_p,G_1^*,G_2^*)\in \cF$ be an arbitrary tuple. Then, the following statements hold:
\begin{itemize}

\item[{\rm (a)}] Let $v \in V(G_2^*)$ be an arbitrary vertex. Then, there is a hypergraph $H \subseteq G_1^* + v$ with $d_H(w) \geq f_1(w)$ for all $w \in V(H)$ and each such hypergraph contains the vertex $v$. As as a consequence, $d_{G_2^*}(v) \leq f_2(v) + f_3(v) + \ldots + f_p(v)=h(v)$ for all $v \in V(G_2^*)$.
\item[{\rm (b)}] The hypergraph $G_2^*$ is not $f'$-partitionable and any non-$f'$-partitionable component $K$ of $G_2^*$ is $h$-regular and contains a vertex $v^*$ from $G_1$.
\item[{\rm (c)}]  Let $K$ be a non $f'$-partitionable component of $G_2^*$ and let $v^* \in V(K) \cap V(G_1)$. Moreover, let  $H \subseteq G_1^* + v^*$ be a hypergraph with $d_H(w) \geq f_1(w)$ for all $w \in V(H)$. Then, $H$ contains a vertex $w^*$ from $V(G) \setminus V(G_1)$.
\item[{\rm (d)}] Let $K$ be a non $f'$-partitionable component of $G_2^*$ and let $v^* \in V(K) \cap V(G_1)$. Moreover, let  $H \subseteq G_1^* + v^*$ be a hypergraph with $d_H(w) \geq f_1(w)$ for all $w \in V(H)$ and let $u^*$ be a vertex that is adjacent to $v^*$ in $H$. Then, $\tilde{G_1}=G_1^*+v^*-u^*$ is a maximum order strictly $f_1$-degenerate subhypergraph of $G$ and with $\tilde{G_2}=G_2^*+u^*-v^*$ we have $(G_1,G_2,\ldots,G_p,\tilde{G_1},\tilde{G_2})\in \mathcal{F}$. Furthermore, $\tilde{G_2}$ has at most as many non $f'$-partitionable components as $G_2^*$ and if equality holds, then $u^*$ is contained in a non-$f'$-partitionable component of $\tilde{G_2}$.
\end{itemize}
\end{claim}

\pfcl{For the proof of (a) let $v \in V(G_2^*)$ be an arbitrary vertex.  Since $G_1^*$ is a maximum order strictly $f_1$-degenerate subhypergraph, $G_1^* + v$ is not strictly $f_1$-degenerate and, thus, there is a subhypergraph $H$ of $G_1^* + v$ such that $d_H(w) \geq f_1(w)$ for all $w \in V(H)$. As $G_1$ is strictly $f_1$-degenerate, $H$ contains the vertex $v$ and so $d_{G_1^*}(v) \geq d_H(v) \geq f_1(v)$. As $d_{G_1^*}(v) + d_{G_2^*}(v) \leq d_G(v) \leq f_1(v) + f_2(v) + \ldots + f_p(v)$, this implies that $d_{G_2^*}(v) \leq f_2(v) + f_3(v) + \ldots + f_p(v)$, which proves statement (a). 

 For the proof of (b) assume that $G_2^*$ admits an $f'$-partition $(G_2',G_3',\ldots G_p')$. Then, the tuple $(G_1^*,G_2',G_3',\ldots,G_p')$ is an $f$-partition of $G$ such that $G_1^*$ is a maximum order strictly $f_1$-degenerate subhypergraph of $G$, contradicting the assumption that the lemma is wrong. Hence, $G_2^*$ is not $f'$-partitionable, i.e., $G_2^*$ has at least one non $f'$-partitionable component. Now let $K$ be a component of $G_2^*$ that is not $f'$-partitionable. Then, by (a) and by Proposition~\ref{prop:strict=partitionable}, $d_K(v)=d_{G_2^*}(v)=f_2(v) + f_3(v) + \ldots + f_p(v)$ for all $v \in V(K)$, i.e. $K$ is $h$-regular. As $G - V(G_1)$ is $f'$-partitionable, $K$ clearly contains a vertex $v^*$ from $G_1$. This proves (b). 

For the proof of (c) and (d), let $H \subseteq G_1^* + v^*$ be a hypergraph with $d_H(w) \geq f_1(w)$ for all $w \in V(H)$ (which exists by (a)). By (a), $H$ contains the vertex $v^*$. As $G_1$ is strictly $f_1$-degenerate, $H$ contains a vertex $w^*$ from $V(G) \setminus V(G_1)$, which proves (c). Now let $u^*$ be a vertex that is adjacent to $v^*$ in $H$. Then, $d_{G_2^*}(v^*)=d_K(v^*)=f_2(v^*) + f_3(v^*) + \ldots + f_p(v^*)$ (by (b)), $d_{G_1^*}(v^*) \geq d_H(v^*) \geq f_1(v^*)$, and $d_{G_1^*}(v^*) + d_{G_2^*}(v^*) \leq d_{G}(v^*) \leq f_1(v^*) + f_2(v^*) + \ldots + f_p(v^*)$. As a consequence, we have $d_{G_1^*}(v^*)=f_1(v^*)$ and so $d_{G_1^*}(v^*)=d_H(v^*)$. Hence, $d_{G_1^* - u^*}(v^*) < f_1(v^*)$. As $G_1^* - u^* \subseteq G_1^*$ and $G_1^*$ is strictly $f_1$-degenerate, this implies that $G_1^* + v^* - u^*$ is strictly $f_1$-degenerate as well and so $\tilde{G_1}=G_1^* + v^* - u^*$ is a maximum order strictly $f_1$-degenerate subhypergraph of $G$. Note that $K-v^*$ is $f'$-partitionable (as $K$ is $h$-regular by (b) and by Proposition~\ref{prop:strict=partitionable}) and so $G_2^* - v^*$ has one non $f'$-partitionable component less than $G_2^*$. Clearly, $\tilde{G_2}=G_2^* - v^* + u^*$ may have only one more non $f'$-partitionable component than $G_2^* - v^*$ and if so, $u^*$ must be contained in this component. Since $\tilde{G_1}$ is a maximum order strictly $f_1$-degenerate subhypyergraph of $G$, $(G_1,G_2,\ldots,G_p,\tilde{G_1},\tilde{G_2}) \in \mathcal{F}$ and  the proof is complete.}

Let $(G_1,G_2, \ldots, G_p,G_1^*,G_2^*)\in \cF$ be an arbitrary tuple. Since we assume that the lemma is false, $|G_1|<|G_1^*|$. By Claim~\ref{claim_tuple}(b), $G_2^*$ is not $f'$-partitionable and so there is a non $f'$-partitionable component of $G_2^*$. Let $\mathcal{K}_{(G_1,G_2,\ldots,G_p,G_1^*,G_2^*)}$ denote the set of non $f'$-partionable components of $G_2^*$. Then, by Claim~\ref{claim_tuple}(c), for any $K \in \mathcal{K}_{(G_1,G_2,\ldots,G_p,G_1^*,G_2^*)}$ we have $V(K) \cap V(G_1) \neq \varnothing$. Let $$V_{(G_1,G_2,\ldots,G_p,G_1^*,G_2^*)}=\bigcup_{K \in \mathcal{K}_{(G_1,G_2,\ldots,G_p,G_1^*,G_2^*)}}(V(K) \cap V(G_1)).$$ Moreover, let $\mathcal{T}_{(G_1,G_2,\ldots,G_p,G_1^*,G_2^*)}$ denote the set of all tupels $(v^*,H,w^*)$ such that $v^* \in V_{(G_1,G_2,\ldots,G_p,G_1^*,G_2^*)}$, $H$ is a subhypergraph of $G_1^* + v^*$ with $d_H(w) \geq f_1(w)$ for all $w \in V(H)$ and $w^* \in V(H) \setminus V(G_1)$. By Claim~\ref{claim_tuple}(a),(c), each vertex $v^* \in V_{(G_1,G_2,\ldots,G_p,G_1^*,G_2^*)}$ is contained in some tuple from $\mathcal{T}_{(G_1,G_2,\ldots,G_p,G_1^*,G_2^*)}$.

Now we choose  $(G_1,G_2, \ldots, G_p,G_1^*,G_2^*) \in \mathcal{F}$ such that

\begin{itemize}
\item[(1)] $|G_1 \cap G_1^*|$ is maximum.

\item[(2)] $|\mathcal{K}_{(G_1,G_2,\ldots,G_p,G_1^*,G_2^*)}|$ is minimum subject to (1).

\item[(3)] $m=\min \{\text{dist}_H(v^*,w^*) ~|~ (v^*,H,w^*) \in \mathcal{T}_{(G_1,G_2,\ldots,G_p,G_1^*,G_2^*)}\}$ is minimum subject to (1),(2).
\end{itemize}

Let $(v^*,H,w^*) \in \mathcal{T}_{(G_1,G_2,\ldots,G_p,G_1^*,G_2^*)}$ such that $d_H(v^*,w^*)=m$. If $m=1$, then $w^*$ is in $H$ adjacent to $v^*$ and it follows from Claim~\ref{claim_tuple}(d) that $\tilde{G_1}=G_1^* + v^* - w^*$ is a maximum order strictly $f_1$-degenerate subgraph of $G$. Moreover, $|V(G_1) \cap V(\tilde{G_1})| > |V(G_1) \cap V(G_1^*)|$, contradicting (1). Hence, $m \geq 2$. Let $u^*$ be a vertex that is adjacent to $v^*$ in $H$ and is contained in a shortest $(v^*,w^*)$-hyperpath of $H$. As $m \geq 2$ and by (3), $u^* \in V(G_1)$. By Claim~\ref{claim_tuple}(d), $\tilde{G_1}=G_1 + v^* - u^*$ is a maximum order strictly $f_1$-degenerate subhypergraph of $G$ and $\tilde{G_2}=G_2 + u^* - v^*$ has at most $|\mathcal{K}_{(G_1,G_2,\ldots,G_p,G_1^*,G_2^*)}|$ non $f'$-partitionable components. By (2), $\tilde{G_2}$ has exactly $|\mathcal{K}_{(G_1,G_2,\ldots,G_p,G_1^*,G_2^*)}|$ non $f'$-partitionable components implying (by Claim~\ref{claim_tuple}(d)) that $u^*$ is contained in a non $f'$-partitionable component $K$ of $\tilde{G_2}$.  Then, $(G_1,G_2,\ldots,G_p,\tilde{G_1},\tilde{G_2}) \in \mathcal{F}$ is a tuple satisfying (1) and (2) and $(u^*,H,w^*) \in \mathcal{T}_{(G_1,G_2,\ldots,G_p,\tilde{G_1},\tilde{G_2})}$ with $d_H(u^*,w^*) < d_H(v^*,w^*)=m$, contradicting (3). This proves the lemma.
\end{proof}

\begin{lemma}
Let $G$ be a hypergraph and let $f\in \cV_p(G)$ be a vector function of $G$ with $p\geq 2$ such that $f_1(v)+f_2(v)+\cdots +f_p(v)\geq d_G(v)$ for all $v\in V(G)$. If $G$ is $f$-partitionable, then there is a partition $(G_1,G_2, \ldots, G_p)$ of $G$ such that $G_1$ is a maximum order strictly $f_1$-degenerate subhypergraph of $G$, and for $i \in \{2,3,\ldots,p-1\}$, the hypergraph $G_i$ is a maximum order strictly $f_i$-degenerate subhypergraph of $G-(V(G_1) \cup V(G_2) \cup \cdots \cup V(G_{i-1}))$.
\end{lemma}

\begin{proof}
It follows from Lemma~\ref{lemma:partitio=better} that $G$ has an $f$-partition $(G_1,G_2,\ldots,G_p)$ such that $G_1$ is a maximum order strictly $f_1$-degenerate subhypergraph. Let $G'=G-V(G_1)$. We claim that $f_2(v) + f_3(v) + \ldots + f_p(v) \geq d_{G'}(v)$ for all $v \in V(G')$. Otherwise, $ f_2(v) + f_3(v) + \ldots + f_p(v) < d_{G'}(v)$ for some $v \in V(G')$ and, as $f_1(v) + f_2(v) + \ldots + f_p(v) \geq d_G(v)$, we conclude $d_{G_1}(v) < f_1(v)$. As a consequence, $G_1 + v$ is a strictly $f_1$-degenerate subhypergraph of $G$ with $|G_1 + v| > |G_1|$, contradicting the maximality of $G_1$. Hence, $f_2(v) + f_3(v) + \ldots + f_p(v) \geq d_{G'}(v)$ for all $v \in V(G')$. Let $f'=(f_2,f_3,\ldots,f_p)$. Since $(G_2,G_3,\ldots,G_p)$ is an $f'$-partition of $G'$, we can again apply Lemma~\ref{lemma:partitio=better} and obtain an $f'$-partition $(G_2',G_3',\ldots,G_p')$ of $G'$ such that $G_2'$ is a maximum order strictly $f_2$-degenerate subhypergraph. By repeated application of the above arguments we finally obtain an $f$-partition as required. 
\end{proof}

Clearly, Theorem~\ref{Theorem:Hauptsatz:neu} is a direct consequence of Theorem~\ref{Theorem:Hauptsatz:alt} and the above lemma and so the proof is complete. The next corollary can be deduced easily from Theorem~\ref{Theorem:Hauptsatz:alt} and Theorem~\ref{Theorem:Hauptsatz:neu}.

\begin{corollary}
Let $G$ be a connected hypergraph having maximum degree $\Delta \geq 1$. Moreover, let $d_1, d_2, \ldots, d_p$ be positive integers, $p \geq 2$, such that $d_1 + d_2 + \ldots + d_p \geq \Delta.$ Then, there is a partition $(G_1,G_2,\ldots,G_p)$ of $G$ such that $G_1$ is a maximum order subhypergraph of $G$ with $\col(G_1) \leq d_1$, and for $i \in \{2,3,\ldots,p-1\}$, the hypergraph $G_i$ is a maximum order subhypergraph of $G - ((V(G_1)) \cup V(G_2) \cup \cdots \cup V(G_{i-1}))$ with $\col(G_i) \leq d_i$, unless  $G$ is a $tK_n$ for some $t,n \geq 1$, $d_i=tn_i$ for some $n_i \geq 1$,  $i \in \{1,2,\ldots,p\}$, and $d_1 + d_2 + \ldots + d_p=t(n-1)=\Delta$, or $G=tC_n$ for $t \geq 1$ and $n \geq 3$ odd, $p=2$, and $d_i=t$ for $i \in \{1,2\}$.
\end{corollary}

\end{document}